\documentclass[12pt]{amsart}
\usepackage{graphicx}
\usepackage{esint}
\usepackage{amssymb, mathrsfs, url, amsfonts, amsthm, amsmath}
\usepackage{enumerate}
\usepackage[dvipsnames]{xcolor}
\usepackage{etoolbox}
\apptocmd{\sloppy}{\hbadness 10000\relax}{}{}
\apptocmd{\sloppy}{\vbadness 10000\relax}{}{}
\usepackage{ctable}
\usepackage[justification=justified]{caption}
\usepackage{hyperref}
\usepackage[letterpaper,margin=1.1in]{geometry}

\newtheorem{theorem}{Theorem}[section]
\newtheorem{lemma}[theorem]{Lemma}
\newtheorem{corollary}[theorem]{Corollary}

\theoremstyle{definition}
\newtheorem{definition}[theorem]{Definition}

\theoremstyle{remark}
\newtheorem{remark}[theorem]{Remark}

\let\inf\relax \DeclareMathOperator*\inf{\vphantom{p}inf}

\newcommand{\side}{\mathop\mathrm{side}\nolimits}
\newcommand{\vol}{\mathop\mathrm{vol}\nolimits}

\newcommand{\Haus}{\mathcal{H}}
\newcommand{\Net}{\mathcal{M}}
\newcommand{\RR}{\mathbb{R}}
\newcommand{\ZZ}{\mathbb{Z}}
\numberwithin{equation}{section}
\numberwithin{figure}{section}
\numberwithin{table}{section}

\newcommand{\interior}{\mathop\mathrm{int}\nolimits}
\newcommand{\diam}{\mathop\mathrm{diam}\nolimits}
\newcommand{\Child}{\mathop\mathsf{Child}\nolimits}
\newcommand{\gap}{\mathop\mathrm{gap}\nolimits}

\title[Lower bounds on Bourgain's constant]{Lower bounds on Bourgain's constant\\ for harmonic measure}
\date{November 10, 2022}
\author{Matthew Badger}
\author{Alyssa Genschaw}
\thanks{M.~Badger was partially supported by NSF DMS grant 2154047.}
\address{Department of Mathematics\\ University of Connecticut\\ Storrs, CT 06269-1009}
\email{matthew.badger@uconn.edu}
\address{Mathematics Department\\ Milwaukee School of Engineering\\ Milwaukee, WI 53202}
\email{genschaw@msoe.edu}
\begin{document}

\begin{abstract} For every $n\geq 2$, Bourgain's constant $b_n$ is the largest number such that the (upper) Hausdorff dimension of harmonic measure is at most $n-b_n$ for every domain in $\RR^n$ on which harmonic measure is defined. Jones and Wolff (1988) proved that $b_2=1$. When $n\geq 3$, Bourgain (1987) proved that $b_n>0$ and Wolff (1995) produced examples showing $b_n<1$. Refining Bourgain's original outline, we prove that \[ b_n\geq c\,n^{-2n(n-1)}/\ln(n)\] for all $n\geq 3$, where $c>0$ is a constant that is independent of $n$. We further estimate $b_3\geq 1\times 10^{-15}$ and $b_4\geq 2\times 10^{-26}$.\end{abstract}

\maketitle

\section{Introduction}

An outstanding problem on the boundary behavior of harmonic functions in space is to identify the maximal \emph{minimal dimension} of a subset of the boundary of a domain through which Brownian motion first exits the domain almost surely \cite{bishop-review}. Adopting the parlance of geometric measure theory \cite{Falconer,measure-dim}, one would like to identify the largest possible \emph{(upper) Hausdorff dimension} $\overline{\dim}_H\,\omega$ of harmonic measure $\omega=\omega_\Omega$ across all connected, open sets $\Omega\subsetneq\RR^n$. For planar domains ($n=2$), the Hausdorff dimension of harmonic measure is at most $1$ \cite{Jones-Wolff} (also see \cite{oksendal81,kw-snowflake,carleson-cantor,Makarov,Wolff-sigma}) and this is sharp (e.g.~when $\Omega$ is a disk). For space domains ($n\geq 3$), the Hausdorff dimension of harmonic measure is at most $n-b_n$ for some undetermined value $0<b_n<1$ \cite{Bourgain,Wolff} (also see \cite{LVV}), which we call \emph{Bourgain's constant}. Determining the value of Bourgain's constant is related to understanding certain physical phenomenon (passivation, fouling, poisoning) \cite{filoche-passivation}. While a conjectural best value $1/(n-1)$ for $b_n$ has been proposed \cite{Bishop-questions}, there appears to have been little progress in affirming or disproving this conjecture to date. A numerical experiment \cite{grebenkov-numerical-experiment} suggests that $b_3\leq 0.995$, but this bound has not been mathematically verified.

Very recently, the authors proved an analogue of Bourgain's theorem in the setting of the heat equation \cite{dim-caloric}. In particular, the Hausdorff dimension of caloric measure on any domain in $\RR^n\times\RR$ (space $\times$ time) with the parabolic distance is at most $n+2-\beta_n$ for some $\beta_n>0$; moreover, $\beta_n\leq b_n$ for all $n$. (Thus, we indirectly obtain $b_n>0$.) In the process of establishing this extension, the authors initially had some difficulty implementing the demonstration of $b_3>0$ from \cite{Bourgain} in general dimensions.\footnote{In particular, the ad hoc choice of several parameters in \cite[Lemmas 1 and 2]{Bourgain} is unexplained and hides dimension dependence.} One goal of this paper is to record a complete and direct proof that $b_n>0$ when $n\geq 3$. Further, by refining Bourgain's original outline, we provide quantitative lower bounds on $b_n$ for arbitrary $n\geq 3$ and the first explicit numerical lower bounds on $b_n$ in dimensions $n=3,4$.

\begin{theorem} \label{main} There exists $c>0$ such that $b_n\geq c\,n^{-2n(n-1)}/\ln(n)$ for all $n\geq 3$.
\end{theorem}

\begin{theorem}[15 decimals, 26 decimals]\label{lowdim} We have $b_3\geq 1\times 10^{-15}$ and $b_4\geq 2\times 10^{-26}$. That is to say, the Hausdorff dimension of harmonic measure in $\RR^3$ is at most $$2.99999\,99999\,99999.$$ The Hausdorff dimension of harmonic measure in $\RR^4$ is at most $$3.99999\,99999\,99999\,99999\,999998.$$\end{theorem}

Bourgain's underlying idea to prove estimates of this type is that if the boundary of a domain is spread uniformly in space throughout a cube, then the probability that a Brownian traveler exits the domain near the center of the cube should be smaller than the probability of exiting the domain near the boundary of the cube. This sets up an alternative: at every $m$-adic location, for some integer $m\gg_n 3$, either the boundary has small $(n-\rho)$-dimensional content for some $0<\rho\ll_{m,n} 1$ or the harmonic measure has positive $(n-\lambda)$-dimensional density on a subset of positive measure for some $0<\lambda\ll_{m,n} 1$. It follows that $\overline{\dim}_H\,\omega \leq n-\rho\lambda/(\rho+\lambda)$; see Lemma \ref{dim-lemma} below. This approach exhibits strong dimension dependence in the form of volume concentration near the boundary of balls or cubes in high dimensions. In several places, the arguments in \cite{Bourgain} and \cite{dim-caloric} use continuity and soft analysis to assert the existence of the parameters $m$, $\rho$, and $\lambda$ to conclude that $b_n>0$ and $\beta_n>0$. A challenge in proving Theorem \ref{main} is to carry out explicit estimates wherever feasible. For Theorem \ref{lowdim}, one need only track estimates for a small set of values of $m$, because $\rho\rightarrow 0$ rapidly as $m\rightarrow\infty$. We remark that the method of proof generates some interesting arithmetic phenomenon; see the nonlinear dependence on size of the $m$-adic grid in the lower bounds on Bourgain's constant in Table \ref{table-of-computations}.

In \S\ref{sec:gmt}, we introduce necessary elements from geometric measure theory, including a tight version of Frostman's lemma for sets in $\RR^n$ of Hausdorff dimension $s>n-1$. In \S\ref{sec:estimates}, we revisit and sharpen Bourgain's estimate on harmonic measure from below in terms of the relative size of the local $m$-adic net content. We also record some basic estimates on harmonic measure inside nested rectangles. In \S\ref{sec:parameters}, we reduce the problem of finding lower bounds on $b_n$ to showing the existence of an admissible set of parameters. Using all of these ingredients, we prove Theorem \ref{main} in \S\ref{sec:main} and Theorem \ref{lowdim} in \S\ref{sec:low}. We end with a brief note on future work in \S\ref{sec:coda}.

\emph{Acknowledgements.} The authors thank Max Engelstein for bringing the work \cite{grebenkov-numerical-experiment} by Grebenkov \emph{et al.} and \cite{filoche-passivation} by Filoche \emph{et al.}~to their attention. The authors would also like to thank Polina Perstneva and several anonymous referees for their critiques of an earlier draft of the paper.

\section{Net measures, Hausdorff dimension, and Frostman's lemma} \label{sec:gmt}

Background on Hausdorff measures, net measures, and related topics can be found in \cite{Mattila, Rogers,Bishop-Peres}.
Let $\RR^n$ be equipped with the standard Euclidean distance, i.e.~$|X-Y|=(\sum_1^n (x_i-y_i)^2)^{1/2}$ for all points $X=(x_1,\dots,x_n)$ and $Y=(y_1,\dots, y_n)$. Let $\diam E=\sup\{|X-Y|:X,Y\in E\}$ denote the \emph{diameter} of $E\subset\RR^n$. For each integer $m\geq 2$, define the system $\Delta^m(\RR^n)$ of \emph{half-open $m$-adic cubes} to be all sets $Q$ of the form $$Q=\left[\frac{j_1}{m^k},\frac{j_1+1}{m^k}\right)
\times \cdots \times \left[\frac{j_n}{m^k},\frac{j_{n}+1}{m^k}\right)\quad(k,j_1,\dots,j_{n}\in\ZZ);$$ we say that $Q$ belongs to \emph{generation $k$} of $\Delta^m(\RR^{n})$ and has \emph{side length} $\side Q=m^{-k}$ and \emph{volume} $\vol Q=m^{-kn}$. Note that $\diam Q=\sqrt{n} \side Q$ for every $Q\in\Delta^m(\RR^{n})$.

The cubes in each generation of $\Delta^m(\RR^{n})$ partition $\RR^{n}$. Every cube $Q\in\Delta^m(\RR^{n})$ of generation $k$ is contained in a unique cube $Q^\uparrow\in\Delta^m(\RR^{n})$ of generation $k-1$; we call $Q^\uparrow$ the \emph{parent} of $Q$ and call $Q$ a \emph{child} of $Q^\uparrow$. Extending this metaphor, we may refer to the ancestors ``above'' a cube and descendents ``below'' a cube. For every $Q\in\Delta^m(\RR^n)$, $\#\Child(Q)=m^{n}$ and $\vol Q=\sum_{R\in\Child(Q)}\vol R,$ where $\Child(Q)$ is the set of children.

\begin{definition}Fix $\Delta=\Delta^m(\RR^{n})$ and let $s\in[0,\infty)$. For all $\delta\in(0,\infty]$, we define $$\Net^s_\delta(E)=\inf\left\{\sum_1^\infty (\side E_i)^s:E\subseteq\bigcup_1^\infty E_i, \side E_i\leq \delta, E_i\in\Delta\right\}\quad\text{for all }E\subset\RR^{n};$$ we call $\Net^s_\infty$ the \emph{$s$-dimensional net content} on $\RR^{n}$. We define the \emph{$s$-dimensional net measure} on $\RR^{n}$ by $\Net^s(E)=\lim_{\delta\downarrow 0} \Net^s_\delta(E)$ for all $E\subset\RR^{n}$.\end{definition}

\begin{remark}\label{hausdorff-net} The net contents $\Net^s_\infty$ are outer measures on $\RR^n$ and the net measures $\Net^s$ are Borel regular outer measures on $\RR^{n}$. Unlike the Hausdorff measures $\Haus^s$ and contents $\Haus^s_\infty$, the net measures and  net contents are neither translation nor dilation invariant. Of course, for all $n\geq 1$ and $s>0$, we have $\Haus^s \lesssim_{n} \Net^s\lesssim_{n,m} \Haus^s$ and $\Haus^s_\infty \lesssim_{n} \Net^s_\infty \lesssim_{n,m} \Haus^s_\infty$. This allows us to define Hausdorff dimension using net measures in lieu of Hausdorff measures. While the net measures and net contents depend on the choice of the underlying grid $\Delta=\Delta^m(\RR^n)$, for simplicity we suppress this from the notation. We choose the letter $\Net$ to suggest $m$-adic. An elementary fact is that $\Net^s(E)=0$ if and only if $\Net^s_\infty(E)=0$. \end{remark}

\begin{definition} Let $E\subset\RR^n$. The \emph{Hausdorff dimension} of $E$ is the unique number $\dim_H E\in [0,n]$ where one witnesses a transition from $\Net^s(E)=\infty$ for all $s<\dim_H E$ to $\Net^s(E)=0$ for all $s>\dim_H E$. \end{definition}

\begin{definition} Let $\mu$ be a Borel measure on $\RR^n$. The \emph{upper Hausdorff dimension} of $\mu$ is defined to be $\overline{\dim}_H\,\mu = \inf\{\dim_H E : E\subset \RR^n \text{ is Borel, } \mu(\RR^n\setminus E)=0\}.$
\end{definition}

\begin{lemma}\label{dim-lemma} Fix $\Delta=\Delta^m(\RR^n)$. Let $\mu$ be a Radon measure on $\RR^{n}$ and let $E\subset\RR^{n}$ be a Borel set with $\mu(\RR^{n}\setminus E)=0$. If there exist constants $0<\rho<n$ and $\lambda>0$ such that for every $Q\in\Delta$ with $\side Q\leq 1$, \begin{equation}\label{dim-lemma-1} \Net^{n-\rho}_{m^{-1}\side Q}(E\cap Q)<(\side Q)^{n-\rho}\end{equation} or \begin{equation}\label{dim-lemma-2}\sum_{R\in\Child(Q)} \mu(R)^{1/2} (\vol R)^{1/2} \leq m^{-\lambda}\, \mu(Q)^{1/2}(\vol Q)^{1/2},\end{equation} then $\overline{\dim}_H\,\mu \leq n-\lambda\rho/(\lambda+\rho)$.
\end{lemma}

\begin{proof}[Proof sketch] An implicit version of the lemma with the weaker conclusion $\overline{\dim}_H\,\mu<n$ appears at the end of \cite[pp.~481--483]{Bourgain}. The authors supplied a detailed proof with the indicated upper bound, in the setting of $\RR^{n}\times\RR$, equipped with the parabolic distance. For the Euclidean case, simply repeat the proof of \cite[Theorem 2.10]{dim-caloric}, making the following superficial change. Replace each occurrence of $n+2$ (the Hausdorff dimension of parabolic $\RR^{n}\times\RR$) with $n$ (the Hausdorff dimension of $\RR^n$).
\end{proof}

\begin{lemma}[Frostman's lemma with better constant]\label{frostman} Fix $\Delta=\Delta^m(\RR^n)$. Let $K\subset\RR^{n}$ be a compact set. If $s>n-1$, then there exists a Radon measure $\mu$ on $\RR^{n}$ such that $\mu(Q)\leq (\side Q)^s$ for all $Q\in\Delta$, $\mu(\partial Q)=0$ for all $Q\in\Delta$, and $\mu(\RR^n)=\mu(K)\geq \Net^s_\infty(K)$.\end{lemma}

\noindent\emph{Remark.} When $s>n-1$, the ``constant'' in front of $\Net^s_\infty(K)$ is 1. For smaller $s$, it is possible that $\mu(\partial Q)>0$ for some $Q\in\Delta$ and the proof only gives $\mu(\RR^n)\geq 2^{-n}\Net^s_\infty(K)$. If one would like to relax the requirement that $K$ be compact to $K$ Souslin or to require $\mu$ satisfy the stronger conclusion $\mu(A)\leq (\diam A)^s$ for all sets $A\subset\RR^n$, then the constant in front of $\Net^s_\infty(K)$ becomes even smaller (see \cite[pp.~112--114]{Mattila}).

\begin{proof}Let $K\subset\RR^n$ be compact and fix $s>n-1$. We will modify a standard proof of Frostman's lemma, exploiting the fact that $s$ is greater than the Hausdorff dimension of the boundaries of $m$-adic cubes. Cover $K$ with a finite list of cubes $Q_1,\dots,Q_l\in\Delta$ with the property that each pair $Q_i$ and $Q_j$ have no common ancestors unless $i=j$. Then $\Net^s_\infty(K)=\sum_{i=1}^l \Net^s_\infty (K\cap Q_i)$. Suppose that for each $i$, we can construct a Radon measure $\mu_i$ on $\RR^{n}$ such that $\mu_i(Q)\leq (\side Q)^s$ for all $Q\in\Delta$, $\mu_i(\partial Q)=0$ for all $Q\in\Delta$, and $\mu_i(\RR^n)=\mu_i(K\cap\overline{Q}_i)\geq \Net^s_\infty(K\cap Q_i)$. (See the next paragraph.) Then $\mu=\mu_1+\cdots+\mu_l$ is our desired measure, because each $\mu_i$ vanishes on the boundaries of cubes in $\Delta$ and pairwise $Q_1,\dots, Q_l$ have no common ancestors.

Fix a cube $Q_0\in\Delta$ and assign $E:=K\cap Q_0\subset Q_0$. Following the proof of \cite[Theorem 8.8]{Mattila}, using $m$-adic cubes instead of dyadic cubes, one may produce a sequence of Radon measure $(\nu_k)_{k=1}^\infty$ such that \begin{enumerate}
  \item the support of $\nu_k$ belongs to the closure of $\bigcup\{Q\in \Delta: \side Q =  m^{-k},\; Q\cap E\neq \emptyset\}$;
  \item $\nu_k(\partial Q)=0$ for all $Q\in\Delta$;
  \item $\nu_k(Q)\leq (\side Q)^s$ for all $Q\in\Delta$ with $\side Q\geq m^{-k}$; and,
  \item for each $X\in E$, there is $Q^X\in \Delta$ with $\side Q^X\geq m^{-k}$ and $\nu_k(Q^X)=(\side Q^X)^s$.
\end{enumerate}
By (i), (ii), and (iii), $\nu_k(\RR^n)=\nu_k(\overline{Q_0})=\nu_k(Q_0)\leq (\side Q_0)^s<\infty$ for all $k$ that are large enough so that $\side Q_0 \geq m^{-k}$. By weak compactness of Radon measures, there exists a subsequence $(\nu_{k_j})_{j=1}^\infty$ and a Radon measure $\nu$ such that $\nu_{k_j}$ converges weakly to $\nu$ as $j\rightarrow\infty$ in the sense of Radon measures. By (i), the support of $\nu$ is contained in $K\cap \overline{Q_0}$. By (iv), for each $k$, there exist (maximal) disjoint cubes $Q^{X_1},\dots, Q^{X_p}$ such that $E\subset Q^{X_1}\cup\dots\cup Q^{X_p}$ and $\nu_k(\RR^n)\geq \sum_{i=1}^p \nu_k(Q^{X_i}) = \sum_{i=1}^p (\side Q^{X_i})^s \geq \Net^s_\infty(E)$. Thus, $$\nu(\RR^n)=\nu(\overline{Q_0}) \geq \limsup_{j\rightarrow\infty} \nu_{k_j}(\overline{Q_0})=\limsup_{j\rightarrow\infty} \nu_{k_j}(\RR^n)\geq \Net^s_\infty(E).$$ Next, let $Q\in\Delta$. For all $i\geq 1$, a sufficiently small open neighborhood  of $\partial Q$ can be covered by $C(n) m^{i(n-1)}$ cubes $R\in\Delta$ with $\side R = m^{-i}\side Q$. By (iii), it follows that $$\nu(\partial Q) \leq \liminf_{j\rightarrow\infty} \sum_R \nu_{k_j}(R) \leq C(n) m^{i(n-1-s)}(\side Q)^s\quad\text{for all $i\geq 1$}.$$ Hence $\nu(\partial Q)=0$, since $s>n-1$. Consequently, $\nu(Q)=\lim_{j\rightarrow\infty} \nu_{k_j}(Q) \leq (\side Q)^s$. \end{proof}

\section{Estimates for harmonic measure} \label{sec:estimates}

Harmonic measure is perhaps best viewed through several complementary perspectives,  including geometric function theory see \cite{GM}, potential theory \cite{helms}, and stochastic processes \cite{brownian}. To ease notation, we adopt the following convention. On a bounded domain $\Omega\subset\RR^n$, we let $\partial\Omega$ denote the topological boundary in the Euclidean topology. On an unbounded domain $\Omega\subset\RR^n$, we let $\partial\Omega$ denote the topological boundary in a one-point compactification of $\Omega$ so that $\partial\Omega$ includes the point at infinity. This ensures that for any domain $\Omega\subsetneq\RR^n$, $n\geq 3$, harmonic measure $\omega^X_\Omega$ with pole at $X\in\Omega$ exists and is a Borel probability measure with support in $\partial\Omega$. By Harnack's inequality, $\omega^X_\Omega$ and $\omega^Y_\Omega$ are mutually absolutely continuous for all $X,Y\in\Omega$. In particular, the support and Hausdorff dimension of $\omega^X_{\Omega}$ are independent of the choice of $X$.

Given an $m$-adic grid $\Delta=\Delta^m(\RR^n)$, we let \begin{equation}\label{e:shift} \vec{\Delta}:=\left\{Q+\left(\frac{j_1}{m^{k+1}},\dots,\frac{j_n}{m^{k+1}}\right): Q\in\Delta,\;\side Q=m^{-k},\; j_1,\dots,j_n\in\ZZ\right\}\end{equation} denote the set of translates of $m$-adic cubes that are aligned with $m$-adic cubes of the next generation. The following lemma is modeled after \cite[Lemma 1]{Bourgain}. Results of this type are now collectively referred to as \emph{Bourgain's estimate} and have become a fundamental tool used to study absolute continuity of harmonic measure (see e.g.~\cite{HM-I,7author}). Also see \cite{cantor-dim,Batakis-Zdunik, azzam-dimension-drop} for applications of Bourgain's estimate to study the dimension of harmonic measure on special classes of domains.

\begin{figure}\begin{center}\includegraphics[width=.6\textwidth]{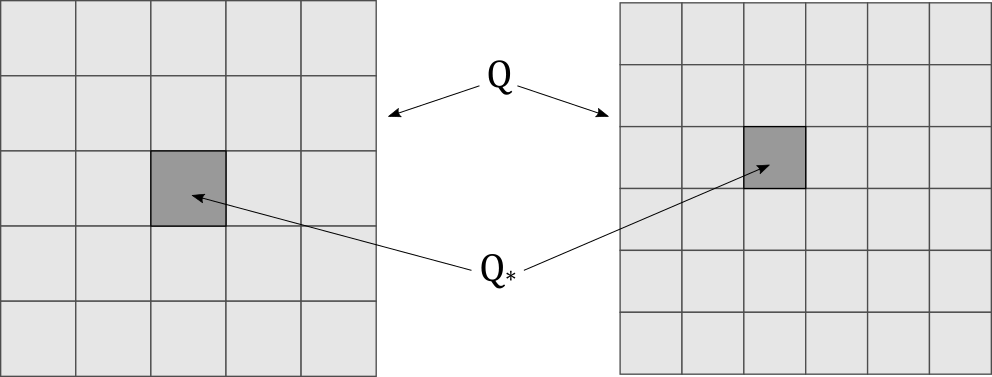}\end{center}\caption{Examples of $Q_*$ and $Q$ in Lemma \ref{alternative} when $m=5$ (left) and $m=6$ (right). The smaller cube $Q_*$ is the closure of a cube $P_*\in\Delta^m(\RR^2)$. The larger cube $Q$ is the interior of a cube $P\in\vec\Delta^m(\RR^2)$.}  \label{fig:odd-and-even}
\end{figure}

\begin{lemma}[Bourgain's estimate with better constants] \label{alternative} Fix $\Delta=\Delta^m(\RR^n)$ for some $n\geq 3$ and $m\geq 5$ with $m>\xi_m+2\sqrt{n}$, where we define $\xi_m:=1$, when $m$ is odd, and $\xi_m:=2$, when $m$ is even. Let $P\in\vec{\Delta}$ and $P_*\in\Delta$ be any cubes such that $\side P_*=m^{-1}\side P$ and $\overline{P_*}$ includes the center of $P$. Assign $Q:=\interior{P}$ and $Q_*:=\overline{P_*}$. For all closed sets $E\subset\RR^{n}$, poles $X\in Q_*\setminus E$, and dimensions $n-1<s\leq n$,
\begin{equation}\label{better-bourgain}
\left(\left(\frac{1}{\sqrt{n}}\right)^{n-2}-\left(\frac{2}{m-\xi_m}\right)^{n-2}\right)\frac{\Net^s_\infty(E\cap Q_*)}{(\side Q_*)^s}\leq O(n,m,s)\,\omega^X_{Q\setminus E}(E\cap Q),\end{equation} where \begin{equation}\label{O-def} O(n,m,s)=\min_{k\in\ZZ}\left(\left(\frac{2m^{k}}{(n-2)\sqrt{n}}\right)^{n-2} +\frac{\omega_n}{4} \left(\frac{mn}{n-2}\right)^{n-2}\frac{n^3m^{k(n-2-s)}}{1-m^{n-2-s}}\right).\end{equation}\end{lemma}

\noindent \emph{Remark.} Here and below, we let $\omega_n$ denote the volume of the unit ball in $\RR^n$. When $m$ is odd, the cube $P_*$ is uniquely determined by $P$. When $m$ is even, there are $2^n$ possibilities for $P_*$ for each $P$. The gap (distance) between $P_*$ and $\partial P$ is $(m-\xi_m)/2$ times $\side P_*$. The value of $k$ that minimizes $O(n,m,s)$ depends on $n$, $\log_n(m)$, and $s$.

\begin{proof} Let $n$, $m$, $Q$, and $Q_*$ be given with the stated requirements. Let $E\subset\RR^{n}$ be closed. Because $Q_*$ is compact, $K:=E\cap Q_*$ is compact. Freeze $n-1<s\leq n$. By Lemma \ref{frostman}, there exists a finite Borel measure $\mu$ on $\RR^{n}$ with support in $K$ such that $\mu(R)\leq (\side R)^s$ for all $R\in\Delta$, $\mu(\partial R)=0$ for all $R\in\Delta$, and $\mu(\RR^n)=\mu(K)\geq \Net^s_\infty(K)$. Consider the harmonic function $u(X)=\int_K |X-Y|^{-(n-2)}\,d\mu(Y)$ defined on $\RR^{n}\setminus K$, which satisfies
\begin{equation*}\tag{E1} \label{estimate1}
    u(X) \leq O(n,m,s)\,(\side Q_*)^{s-(n-2)}
    \quad \forall\,X\in \RR^n\setminus K,
\end{equation*}
\begin{equation*}\tag{E2} \label{estimate2}
    u(X)\geq (\diam Q_*)^{-(n-2)}\mu(K)=\sqrt{n}^{-(n-2)}(\side Q_*)^{-(n-2)}\mu(K)\quad\forall\,X\in Q_*\setminus K,
\end{equation*}
\begin{equation*}\tag{E3} \label{estimate3}
    u(X)\leq (\gap(Q_*,\partial Q))^{-(n-2)}\mu(K)=\left(\frac{m-\xi_m}{2}\side Q_*\right)^{-(n-2)}\mu(K)\quad \forall\,X\in\partial Q,
\end{equation*} where $\gap(A,B)=\inf_{a\in A}\inf_{b\in B}|a-b|$ denotes the \emph{gap} between nonempty sets $A$ and $B$ (sometimes referred to as the \emph{distance} between $A$ and $B$). Of the three estimates, \eqref{estimate2} and \eqref{estimate3} are straightforward; we delay the proof of \eqref{estimate1} to the end of the lemma. Following \cite{Bourgain}, define an auxiliary harmonic function $w$ on $\RR^{n}\setminus K$ by setting $$w(X)=\frac{u(X)-\|u\|_{L^\infty(\partial Q)}}{\|u\|_{L^\infty(\RR^n\setminus K)}}\quad\text{for all }X\in\RR^{n}\setminus K,$$ where the norms denote the supremum of the continuous function $u$ on the specified sets. By design, $w\leq 0$ on $\partial Q$, $w$ is continuous on $\partial Q$, and $w\leq 1$ on all of $\RR^n\setminus K$. Hence $$\limsup_{X\rightarrow X_0} w(X) \leq \chi_K(X_0)\quad\text{for all }X_0\in\partial(Q\setminus K).$$ Thus, $w(X)\leq \omega^{X}_{Q\setminus K}(K)\leq \omega^X_{Q\setminus E}(E\cap Q)$ for every $X\in Q\setminus E$ by two applications of the maximum principle. Suppose that $X\in Q_*\setminus E=Q_*\setminus K$. Using (E1) to estimate the denominator in the definition of $w$ and (E2), (E3) to estimate the numerator, we obtain
\begin{align*}
    &\omega^X_{Q\setminus E}(Q\cap E)\geq \frac{\left((1/\sqrt{n})^{n-2} - (2/(m-\xi_m))^{n-2}\right)(\side Q_*)^{-(n-2)}\Net^s_\infty(K)}{O(n,m,s)\,(\side Q_*)^{s-(n-2)}}.
\end{align*} This is our desired estimate.

\begin{figure}\begin{center}\includegraphics[width=.6\textwidth]{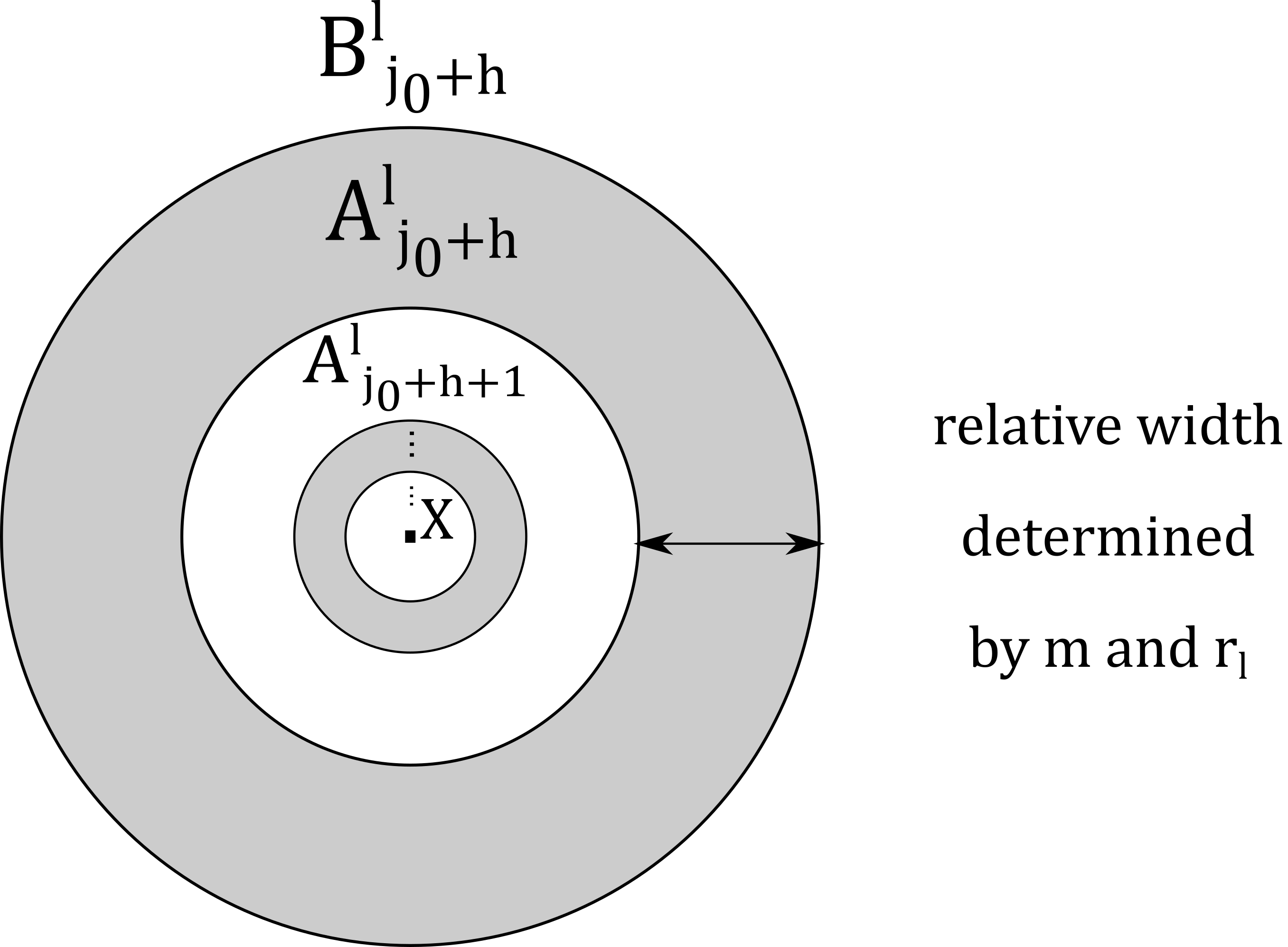}\end{center}\caption{Annular decomposition used in the proof of \eqref{estimate1}.}  \label{fig:annuli}
\end{figure}

It remains to verify \eqref{estimate1}, which is the main improvement over the corresponding lemma in \cite{Bourgain}. We will use an annular decomposition (see Figure \ref{fig:annuli}), but do not guess the geometry of the annuli in advance. Fix $X\in \RR^{n}\setminus K$ and let $j_0$ denote the integer such that $\side Q_*=m^{-j_0}.$ Recall that $\mu(R)\leq (\side R)^s$ for all $R\in\Delta$. Let $l\in\ZZ$ and let $r_l>0$ (depending on $l$) to be determined. Let $h\in\ZZ$ and write $K=B^l_{j_0+h}\cup \bigcup_{j=j_0+h}^\infty A_j^l$, where $$B^l_j:=\left\{Y\in K:|X-Y|\geq r_l m^{-j}\right\},\quad A_j^l:=\left\{Y\in K: r_l m^{-(j+1)}\leq |X-Y|<r_lm^{-j}\right\}.$$ On one hand, we may trivially estimate $\mu(B^l_{j_0+h})\leq \mu(Q_*)\leq m^{-j_0s}$. On the other hand, let $\omega_n$ denote the volume of the unit ball in $\RR^n$. Then $A_j^l$ is covered by $\lfloor \omega_n (r_l m^l+\sqrt{n})^n\rfloor$ or fewer cubes in $\Delta$ of side length $m^{-(j+l)}$. (To derive this, let $R$ represent a cube of side length $m^{-(j+l)}$; divide the volume of a ball of radius $r_lm^{-j}+\diam R$ by the volume of $R$.) Hence $\mu(A_j^l)\leq \omega_n (r_l m^l+\sqrt{n})^n m^{-(j+l)s}$. Therefore, \begin{align*}
u(X) &\leq (1/r_l)^{n-2}m^{(n-2)(j_0+h)}\mu (B^l_{j_0+h}) + \sum_{j=j_0+h}^\infty (1/r_l)^{n-2}m^{(n-2)(j+1)} \mu(A_j^l) \\
     &\leq (m^{h}/r_l)^{n-2} m^{j_0(n-2-s)} + (m/r_l)^{n-2} \omega_n (r_lm^l+\sqrt{n})^n m^{-ls}\sum_{j=j_0+h}^\infty m^{j(n-2-s)} \\
     &\leq \underbrace{\left((m^{h}/r_l)^{n-2} +\omega_n(m/r_l)^{n-2}(r_lm^l+\sqrt{n})^n \frac{m^{-ls+h(n-2-s)}}{1-m^{n-2-s}}\right)}_I(\side Q_*)^{s-(n-2)}.\end{align*} The quantity $(m/r_l)^{n-2}(r_lm^l+\sqrt{n})^n$ has a unique critical point and is minimized across all $r_l>0$ when $r_l=((n-2)\sqrt{n})/2m^l$.
Selecting this value yields \begin{align*} I&=\left(\frac{2m^{h+l}}{(n-2)\sqrt{n}}\right)^{n-2} +\omega_n\left(\frac{2m^{1+l}}{(n-2)\sqrt{n}}\right)^{n-2}\left(\frac{n\sqrt{n}}{2}\right)^n \frac{m^{-ls+h(n-2-s)}}{1-m^{n-2-s}}\\
     &=\left(\frac{2m^{h+l}}{(n-2)\sqrt{n}}\right)^{n-2} +\frac{\omega_n}{4} \left(\frac{mn}{n-2}\right)^{n-2}\frac{n^3m^{(h+l)(n-2-s)}}{1-m^{n-2-s}}.
     \end{align*} Letting $h+l$ range over arbitrary values $k\in\ZZ$, we arrive at \eqref{estimate1}.
\end{proof}

We use the following two special cases to prove Theorem \ref{main}.

\begin{corollary}\label{c-all-dim} For all $n\geq 3$, there exists $C_n>1$ depending only on $n$ such that for all $m>\xi_m+2\sqrt{n}$ and $n-1<s\leq n$, \begin{equation}\label{general-bourgain} \Net^s_\infty(E\cap Q_*)\leq C_n m^{n-2}\,\omega^X_{Q\setminus E}(E\cap Q)\, (\side Q_*)^s.\end{equation}\end{corollary}

\begin{proof} In \eqref{O-def}, take $k=0$ or $k=1$.\end{proof}

\begin{corollary}\label{c-large-dim} For all $\delta>0$, there exists $n_\delta\geq 3$ such that if $n$ and $m$ are integers with $n\geq n_\delta$ and $m\geq \delta n$, and $n-1/2<s\leq n$, then \begin{equation}\label{great-bourgain} \Net^s_\infty(E\cap Q_*)\leq \delta (m\sqrt{2\pi e})^{n-2}\,\omega^X_{Q\setminus E}(E\cap Q)\, (\side Q_*)^s.\end{equation}
\end{corollary}

\begin{proof} Let $\delta>0$ be given and fix $\epsilon>0$ to be specified below. Suppose that $n$ and $m$ are integers with $n\geq 3$ and $m\geq \delta n$. Let us agree to write $c_n\sim d_n$ if $c_n$ is \emph{asymptotic} to $d_n$ as $n\rightarrow\infty$ in the sense that $\lim_{n\rightarrow\infty} c_n/d_n = 1$ and further agree to write $c_n=o(d_n)$ if $\lim_{n\rightarrow\infty} c_n/d_n=0$. Using Stirling's formula for $\Gamma(x)$ (e.g.~\cite{a-proof-of-stirling}) and   $\omega_n=\pi^{n/2}/\Gamma(\frac{n}{2}+1)$ (e.g.~\cite[Appendix A]{HFT}), we see that $$\omega_n\sim \frac{1}{\sqrt{\pi n}} \left(\frac{\sqrt{2\pi e}}{\sqrt{n}}\right)^n=\frac{2\sqrt{\pi}e}{\sqrt{n}^3}\left(\frac{\sqrt{2\pi e}}{\sqrt{n}}\right)^{n-2}.$$ We also have $(n/(n-2))^{n-2}\sim e^2$.  Choosing $k=1$ in \eqref{O-def}, it follows from the noted asymptotic estimates that for sufficiently large $n$, depending only on $\epsilon$, \begin{align*}
\frac{O(n,m,s)}{(1+\epsilon)e^2} &\leq \left(\frac{2m}{n\sqrt{n}}\right)^{n-2} + \frac{\sqrt{\pi}e}{2} \sqrt{n}^{3}m^{n-2-s}\left(\frac{m\sqrt{2\pi e}}{\sqrt{n}}\right)^{n-2},\end{align*} where the reader may observe that we also absorbed the factor $(1-m^{n-2-s})^{-1}$ appearing in $O(n,m,s)$ into the error on the left hand side. Because $m\geq \delta n$ and $-2\leq n-2-s<-3/2$, the factor $\sqrt{n}^3 m^{n-2-s}=o(\max\{\delta^{-2},\delta^{-3/2}\})=o(1)$. Also, $(2/n)^{n-2} = o(\sqrt{2\pi e}^{\,n-2})$. Thus, taking $n$ to be sufficiently large depending only on $\delta$ and $\epsilon$, we have \begin{equation}\label{apple1} O(n,m,s) \leq \epsilon(1+\epsilon)e^2\left(\frac{m\sqrt{2\pi e}}{\sqrt{n}}\right)^{n-2}.\end{equation} Next, because $m\geq \delta n$, we can estimate \begin{equation}\begin{split} \label{apple2} \left(\frac{1}{\sqrt{n}}\right)^{n-2}-\left(\frac{2}{m-\xi_m}\right)^{n-2}\geq (1+\epsilon)^{-1}\sqrt{n}^{-(n-2)}\end{split}\end{equation} for all sufficiently large $n$ depending only on $\epsilon$ and $\delta$. Combining \eqref{better-bourgain}, \eqref{apple1}, and \eqref{apple2}, we conclude that for all $n$ sufficiently large depending only on $\epsilon$ and $\delta$, \begin{align*}\Net^s_\infty(E\cap Q_*)&\leq \epsilon(1+\epsilon)^2 e^2(m\sqrt{2\pi e})^{n-2} \omega^X_{Q\setminus E}(E\cap Q)(\side Q_*)^s.\end{align*} Specifying that $\epsilon(1+\epsilon)^2 e^2=\delta$ yields \eqref{great-bourgain}.
\end{proof}

We use the next two special cases to prove Theorem \ref{lowdim}.

\begin{corollary}\label{c-3} Suppose $n=3$. For all integers $m\geq 5$ and dimensions $2.999999\leq s\leq 3$, \begin{equation}\left(\frac{1}{\sqrt{3}}-\frac{2}{m-\xi_m}\right)\frac{\Net^s_\infty(E\cap Q_*)}{(\side Q_*)^s}\leq \left(\frac{2}{\sqrt{3}}m + 27\pi\frac{m^{-0.999999}}{1-m^{-1.999999}}\right)\omega^X_{Q\setminus E}(E\cap Q).\end{equation}\end{corollary}

\begin{proof} In \eqref{O-def}, take $n=3$ and $k=1$. Bound the factor $m^{2-s}/(1-m^{1-s})$ appearing in $O(3,m,s)$ using the assumption $s\geq 2.999999$. \end{proof}

\begin{corollary}\label{c-4} Suppose $n=4$. For all integers $m\geq 7$ and dimensions $3.999999\leq s\leq 4$, \begin{equation}\left(\frac{1}{4}-\left(\frac{2}{m-\xi_m}\right)^2\right)\frac{\Net^s_\infty(E\cap Q_*)}{(\side Q_*)^s}\leq\left(\frac{1}{4}m^2 + 32\pi^2 \frac{m^{0.000001}}{1-m^{-1.999999}}\right)\omega^X_{Q\setminus E}(E\cap Q).\end{equation} \end{corollary}

\begin{proof} In \eqref{O-def}, take $n=4$ and $k=1$. Bound the factor $m^{4-s}/(1-m^{2-s})$ appearing in $O(4,m,s)$ using the assumption $s\geq 3.999999$.\end{proof}

A \emph{rectangle} in $\RR^{n}$ is a set of the form $[x_1,x_1+s_1]\times\cdots\times [x_n,x_n+s_n]$ with $s_1,\dots,s_n>0$. Iterating the strong Markov property, one gets an estimate on harmonic measure of the portion of the boundary lying inside a sequence of nested rectangles. Brownian motion cannot reach the innermost rectangle without passing through the outer rectangles.

\begin{lemma}\label{c-nested} Let $n\geq 3$ and let $\Omega\subsetneq\RR^{n}$ be a domain. Let $H_1,\dots,H_k$ be rectangles in $\RR^{n}$ that are strictly nested in the sense that $H_k\subset\interior{H_{k-1}}$, $H_{k-1}\subset\interior H_{k-2}$, \dots, $H_2\subset\interior{H_1}$. Write $G'_i=\Omega\cap\partial(\Omega\setminus H_i)\subset\partial H_i=G_i$ for each $i$. If $X\in \Omega\setminus H_1$, then \begin{equation}\omega^{X}_\Omega(H_k)\leq \omega^{X}_{\Omega\setminus H_1}(G'_1)\left(\sup_{X_1\in G'_1} \omega^{X_1}_{\Omega\setminus H_2}(G'_2)\right)
\cdots\left(\sup_{X_{k-1}\in G'_{k-1}}\omega^{X_{k-1}}_{\Omega\setminus H_{k}}(G_k)\right).\end{equation} (Except for the final instance, $G_k$, all instances of a `$G$' in the formula are $G'_i$.)\end{lemma}

\begin{proof} We induct on the number of rectangles. The base case $\omega^X_\Omega(H_1)\leq \omega^{X}_{\Omega\setminus H_1}(G_1)$ holds by the maximum principle. For the induction step, suppose that the lemma holds with $k$ nested rectangles for some $k\geq 1$. Let $H_1,\cdots, H_{k+1}$ be rectangles with $H_{j+1}\subset \interior{H_{j}}$ for all $1\leq j\leq k$ and fix $X\in\Omega\setminus H_1$. Note that every $X_1\in G'_1$ lies outside of $H_2$. Thus, the inductive hypothesis applied with $H_2,\dots, H_{k+1}$ guarantees that $$\omega^{X_1}_\Omega(H_{k+1})\leq \omega^{X_1}_{\Omega\setminus H_2}(G'_2)\left(\sup_{X_2\in G'_2} \omega^{X_2}_{\Omega\setminus H_3}(G'_3)\right)
\cdots\left(\sup_{X_{k}\in G'_{k}}\omega^{X_{k}}_{\Omega\setminus H_{k+1}}(G_{k+1})\right).$$ (When $k=1$, this formula should be read as $\omega^{X_1}_\Omega(H_2)\leq \omega^{X_1}_{\Omega\setminus H_2}(G_2)$.) Since $\Omega\setminus H_1\subset \Omega$, the strong Markov property (e.g.~see \cite[p.~117]{Doob}) ensures that \begin{align*}
\omega^{X}_\Omega(H_{k+1}) &= \omega^{X}_{\Omega\setminus H_1}(H_{k+1}\cap\partial\Omega)
   + \int_{\Omega\cap \partial(\Omega\setminus H_1)} \omega^{X_1}_{\Omega}(H_{k+1})\,d\omega^{X}_{\Omega\setminus H_1}(X_1)\\
&= \int_{G'_1} \omega^{X_1}_{\Omega}(H_{k+1})\,d\omega^{X}_{\Omega\setminus H_1}(X_1)
\leq \omega^{X}_{\Omega\setminus H_1}(G'_1)\sup_{X_1\in G'_1} \omega^{X_1}_{\Omega}(H_{k+1}),\end{align*} where $\omega^{X}_{\Omega\setminus H_1}(H_{k+1}\cap\partial\Omega)=0$ trivially, since $H_{k+1}$ is contained in the exterior of $\Omega\setminus H_1$. Combining the two displayed equations gives the desired inequality for $H_1,\dots,H_{k+1}$.\end{proof}

A nearly identical argument gives the following dual inequality.

\begin{lemma}\label{c-lower-bound} Let $n\geq 3$ and let $\Omega\subsetneq\RR^{n}$ be a domain. Let $H_1, H_2$ be rectangles in $\RR^{n}$ with $H_2\subset \interior{H_1}$. Write $G'_i=\Omega\cap \partial(\Omega\setminus H_i)\subset \partial H_i=G_i$ for all $i$. If $X\in \Omega\setminus H_1$, then \begin{equation}\omega^{X}_\Omega(H_2)\geq \omega^X_{\Omega\setminus H_1}(G'_1)\left(\inf_{X_1\in G'_1} \omega^{X_1}_{\Omega\setminus H_2}(G_2)\right).\end{equation}
\end{lemma}

\section{Bounding Bourgain's constant from below}\label{sec:parameters}

Recall that Bourgain's constant $b_n$ is the largest value such that the upper Hausdorff dimension of harmonic measure is at most $n-b_n$ for all domains $\Omega\subset\RR^n$. The following theorem is based on the demonstration in \cite{Bourgain} that $b_3>0$ and implements ideas from \cite{dim-caloric}. It reduces the problem of bounding $b_n$ from below to estimation of constants appearing in Bourgain's estimate and selection of parameters $m$, $\eta$, $h$, and $d$ satisfying the constraint \eqref{gamma-def}. By working exclusively with $m$-adic cubes and net contents---without passing through Hausdorff contents---we avoid introducing an unnecessary source of error as was done in the original argument. This is important in the context of Theorem \ref{lowdim}.

\begin{theorem}\label{t:parameters} Let $n\geq 3$ and let $m\geq 5$. Suppose that $\epsilon>0$ and $\alpha>0$ are constants such that Bourgain's estimate holds in the sense that for all $Q_*$ and $Q$, as in Lemma \ref{alternative}, for all closed sets $E\subset\RR^n$, for all $X\in Q_*\setminus E$, and for all $n-\epsilon<s\leq n$, we have \begin{equation}\label{alpha-bourgain}\Net^s_\infty(E\cap Q_*)  \leq \alpha\, \omega^X_{Q\setminus E}(E\cap Q) (\side Q_*)^s.\end{equation} Let $\eta>0$ be any number such that $(2-m^{-n})\alpha\eta\leq 1-m^{-n}$. Finally, suppose that $1\leq h<m/2$ and $d\geq 1$ are integers such that \begin{equation}\label{gamma-def}\gamma:=(1-(1-2h/m)^n)^{1/2} + (1-2h/m)^{n/2}\eta^{-1/2}(1-\eta)^{hm^{d-1}/2}<1.\end{equation} Then $b_n\geq \lambda\rho/(\lambda+\rho)$, where \begin{equation}\label{lambda-and-rho} \lambda:=-\log_m(\gamma)\quad\text{and}\quad \rho:=\min\{\epsilon,0.914186(1-\alpha\eta)(1-m^{-n})m^{-(d+1)n}/\ln(m)\}.\end{equation}
\end{theorem}

In order to prove the theorem, we start with an auxiliary estimate.

\begin{lemma}\label{prelemma-1} If $n\geq 3$, $m\geq 5$, and $d\geq 1$ are integers and $(2-m^{-n})a\leq 1-m^{-n}$, then \begin{equation}\label{rho-goal} (m^n-1)(m^{\rho-n}+ m^{2(\rho-n)}+\cdots+ m^{(d+1)(\rho-n)})<1-am^{(d+1)(\rho-n)}\end{equation} holds for all values of $\rho$ in the range \begin{equation}\label{rho-stipulation} 0\leq \rho \leq 0.914186(1-a)(1-m^{-n})m^{-(d+1)n}/\ln(m).\end{equation}
\end{lemma}

\noindent\emph{Remark.} As a referee noted, the proof below shows that by imposing stricter constraints on $m$ and $d$, the constant $0.914186$ can be made arbitrarily close to $\sqrt{2}-1/2=0.914213...$. We shall not dwell on this point, because it would not change the first significant digit of our estimate on $b_3$ and $b_4$ in Theorem \ref{lowdim}. It may be worth exploring how much the bound can be improved without using the relaxation $c-x<1\Rightarrow c - c^{d+2}x<1$, but this is beyond the scope of the current paper.

\begin{proof} Rewriting \eqref{rho-goal} using the formula for partial geometric series, we want to find $\rho\geq 0$ (bigger is better) such that \begin{align*}
\frac{(m^{\rho}-m^{\rho-n})(1-m^{(d+1)(\rho-n)})}{1-m^{\rho-n}}< 1-a m^{(d+1)(\rho-n)}.\end{align*} Rearranging, expanding the products, and cancelling like terms, our goal becomes $$m^\rho(1-m^{(d+1)(\rho-n)})+m^{(d+2)(\rho-n)}<1-am^{(d+1)(\rho-n)}+am^{(d+2)(\rho-n)}.$$ Set $\rho=\log_m(c)$ with $c\geq 1$, close to $1$, to be found below. Then we would like $$c-c^{d+2}\underbrace{m^{-(d+1)n}(1-(1-a)m^{-n}-a/c)}_x<1.$$ Since $c\geq 1$ and $x>0$, the inequality $c-c^{d+2}x<1$ is implied by $c-x<1$. Hence it suffices to find $c\geq 1$ such that $c-m^{-(d+1)n}(1-(1-a)m^{-n}-a/c)< 1$. Equivalently, $$c^2-(1+\underbrace{m^{-(d+1)n}(1-(1-a)m^{-n})}_{y})\,c+\underbrace{am^{-(d+1)n}}_z< 0.$$ Now, $c^2-(1+y)c+z< 0$ holds at $c=1$ provided that $z<y$. In our case, we need $(1-a)m^{-n}< 1-a$, which is true since $a<1$. It follows that we may select $c$ to be any number between 1 and the greater of the two roots of $c^2-(1+y)c+z=0$. That is, $$1\leq c < \frac{1+y+\sqrt{(1+y)^2-4z}}{2}.$$ Well, $1+(\sqrt{2}-1)(2y-4z)\leq \sqrt{1+2y-4z}<\sqrt{(1+y)^2-4z}$ provided that $0\leq 2y-4z<1$. (To verify the first inequality, start by squaring both sides.) In particular, $4z\leq 2y$ as long as $4a \leq 2-2(1-a)m^{-n}$; this holds by our demand that $(2-m^{-n})a\leq 1-m^{-n}$. Hence we can choose \begin{equation*}\begin{split} 1\leq c&\leq  \frac{1+y+1+(\sqrt{2}-1)(2y-4z)}{2}\\
&=1+\underbrace{(\sqrt{2}-1/2)m^{-(d+1)n}\left(\frac{(1-2(\sqrt{2}-1))a}{(\sqrt{2}-1/2)-(1-a)m^{-n}}\right)}_w.\end{split}\end{equation*} Thus, \eqref{rho-goal} holds if $0\leq \rho \leq \ln(1+w)/\ln(m)$. Estimating $$\ln(1+w)\geq w(1-(1/2)w) \geq w(1-(1/2)(\sqrt{2}-1/2)\cdot 5^{-6})\geq 0.99997w$$ and $1-2(\sqrt{2}-1)a/(\sqrt{2}-1/2)-(1-a)m^{-n} \geq (1-a)(1-m^{-n})$ and checking that $(\sqrt{2}-1/2)\cdot0.99997=0.914186...$, we conclude that \eqref{rho-stipulation} implies \eqref{rho-goal}.
\end{proof}

\begin{proof}[Proof of Theorem \ref{t:parameters}] Let $n$, $m$, $\epsilon$, $\alpha$, $\eta$, $h$, $d$, $\gamma$, $\lambda$, and $\rho$ be given according to the statement of the theorem. Shrinking $\rho$ as needed, we may assume without loss of generality that $\rho<\epsilon$. Let $\Delta=\Delta^m(\RR^n)$, let $\Omega\subsetneq\RR^n$ be a domain, let $\Omega^c=\RR^n\setminus\Omega$, let $X\in \Omega$, and let $\omega=\omega^{X}_{\Omega}$. Define $\vec\Delta$ as in \eqref{e:shift}. We say $(Q,Q_*)$ is an \emph{admissible pair} if $Q\in\vec\Delta$, $Q_*\in\Delta$, $\side Q_*=m^{-1}\side Q$, and $\overline{Q_*}$ includes the center of $Q$. For every admissible pair $(Q,Q_*)$, the Bourgain type estimate \eqref{alpha-bourgain} with $s=n-\rho$ implies \begin{equation}
\label{alt1} \omega^{Z}_{(\interior{Q})\setminus \Omega^c}(\Omega^c \cap \interior{Q}) \geq \eta\quad\text{for all }Z\in \overline{Q_*}\setminus \Omega^c,\end{equation} or \begin{equation}\label{alt2}\Net^{n-\rho}_\infty(\Omega^c\cap Q_*)< \alpha\eta (\side Q_*)^{n-\rho}.\end{equation}

To bound $\overline{\dim}_H\,\omega$ from above by $n-\lambda\rho/(\lambda+\rho)$, we aim to use Lemma \ref{dim-lemma}. Because scaling and translating the domain in space and changing the pole has no effect on the Hausdorff dimension of harmonic measure, we may assume without loss of generality that if $P\in\Delta$, $\side P\leq 1$, and $X\in \overline{P}$, then $\overline{P}$ is disjoint from $\partial\Omega$. For any such cube $P$, \begin{equation}\label{outcome0} \sum_{Q\in\Child(P)} \omega(Q)^{1/2}(\vol Q)^{1/2}=0= m^{-\lambda} \omega(P)^{1/2}(\vol P)^{1/2},\end{equation} trivially.

To continue, suppose that $P\in\Delta$ is an $m$-adic cube with $\side P\leq 1$, for which $X\not\in \overline{P}$. For any $j\geq 1$, let $\Child^j(P)$ denote the set of all $j$-th generation descendents of $P$ in the tree $\Delta$. For example, $\Child^2(P)=\{R\in\Delta:R\subset P,\, \side R = m^{-2}\side P\}$ is the set of all grandchildren of $P$. Keeping in mind our goal of checking the hypothesis of Lemma \ref{dim-lemma}, we consider two alternatives. Under Alternative 1, we will show that $P$ satisfies \eqref{dim-lemma-1}. Under Alternative 2, we will show that $P$ satisfies \eqref{dim-lemma-2}.

\smallskip

\emph{Alternative 1. Suppose that the estimate \eqref{alt2} holds for some admissible pair $(Q,Q_*)$ with $Q_*\in\Child^{d+1}(P)$.} Let $Q_*^{\uparrow j}\in\Delta$ denote the $j$-th ancestor of $Q_*$ in $\Delta$. Covering $\Omega^c\cap P$ by $\Child(P)\setminus\{Q_*^{\uparrow d}\}$, $\Child(Q_*^{\uparrow d})\setminus\{Q_*^{\uparrow d-1}\}$, \dots, $\Child(Q_*^{\uparrow 1})\setminus \{Q_*\}$, and $\Omega^c\cap Q_*$, we obtain \begin{equation*}\begin{split}\Net^{n-\rho}_{m^{-1}\side P}(\Omega^c\cap P)< (m^{n}-1)\left((\side Q_*^{\uparrow d})^{n-\rho}+\cdots+(\side Q_*)^{n-\rho}\right)+\alpha\eta(\side Q_*)^{n-\rho}.\end{split}\end{equation*} Rewriting each side length in terms of $\side P$ and rearranging, $$\frac{\Net^{n-\rho}_{m^{-1}\side P}(\Omega^c\cap P)}{(\side P)^{n-\rho}} \leq (m^{n}-1)(m^{-(n-\rho)}+\cdots+m^{-(d+1)(n-\rho)})+\alpha\eta m^{-(d+1)(n-\rho)}.$$ Applying Lemma \ref{prelemma-1} with $a=\alpha\eta$, we conclude that \begin{equation}\label{outcome1}\Net^{n-\rho}_{m^{-1}\side P}(\Omega^c\cap P)<(\side P)^{n-\rho}.\end{equation}

\begin{figure}\begin{center}\includegraphics[width=.8\textwidth]{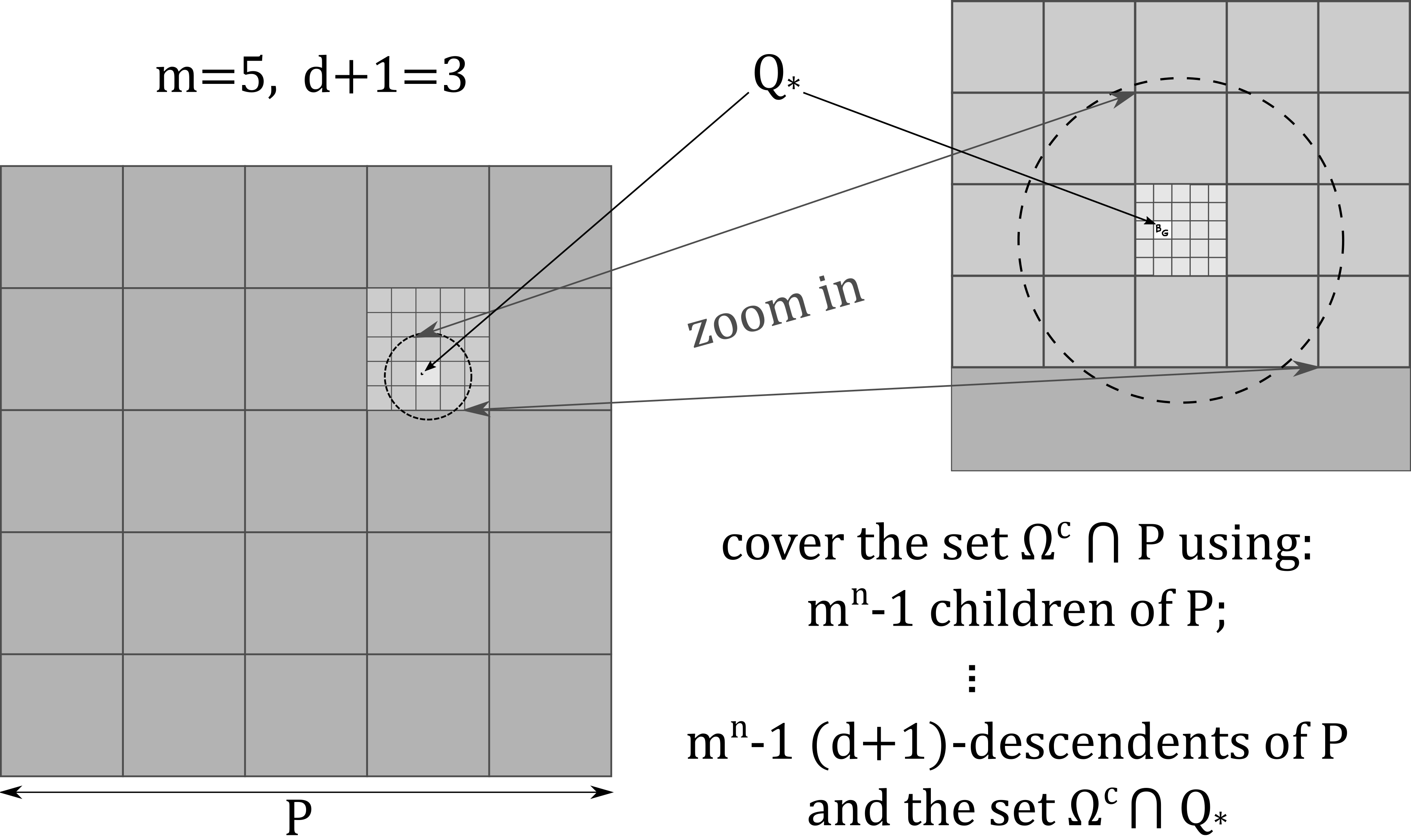}\end{center}\caption{Optimal covering of $\Omega^c\cap P$ under the assumption that the net content of $\Omega^c$ inside of some $(d+1)$-descendent $Q_*$ is small.}  \label{fig:alt1}
\end{figure}

\smallskip

\emph{Alternative 2. Suppose that the estimate \eqref{alt1} holds for every admissible pair $(Q,Q_*)$ with $Q_*\in\Child^{d+1}(P)$.}
Partition $P$ into annular rings of $d$-th generation descendents. Working from the outside to the inside, define $A_0=\emptyset$, $P_0= P$,
\begin{align*}
A_1 &= \bigcup\{Q\in\Child^d(P):Q\not\subset A_0,\, \overline{Q}\cap \partial P_0\neq\emptyset\}, & P_1&=P_0\setminus A_1,\\
A_2 &= \bigcup\{Q\in\Child^d(P):Q\not\subset A_1,\, \overline{Q}\cap \partial P_1\neq\emptyset\}, & P_2&=P_1\setminus A_2,\\
&\ \,\vdots &&\ \,\vdots\\
A_{M} &= \bigcup\{Q\in\Child^d(P):Q\not\subset A_{M-1},\, \overline{Q}\cap \partial P_{M-1}\neq\emptyset\}, & P_{M}&=P_{M-1}\setminus A_M,\\
A_{M+1}&= \bigcup\{Q\in\Child^d(P):Q\not\subset A_{M},\, \overline{Q}\cap \partial P_M \neq \emptyset\}=P_M, && \!\!\!\!\!\!\!\!\!\!P_{M+1}=\emptyset,
\end{align*} where $m^d=2M+1$, if $m$ is odd, and $m^d=2M+2$, if $m$ is even. Next, for each annulus $A_i$, with $1\leq i\leq M$, choose a rectangle $H_i$ such that (i)  $G_i:=\partial H_i \subset A_i$ separates $\partial P_{i-1}$ from $\partial P_i$ and (ii) for any $Z\in G_i$, there exists an admissible pair $(Q,Q_*)$ with $Q_*\in\Child^{d+1}(P)$, $Z\in Q_*$, $Q\subset A_i$, and $Q\cap A_{i+1}=\emptyset$. There are a continuum of possibilities for each $H_i$. Further, as in Lemma \ref{c-nested}, assign $G_i':=\Omega\cap \partial(\Omega\setminus H_i)\subset G_i$ for each $i$. See Figure \ref{fig:paw}.

Fix any $1\leq k\leq M$. Later we will choose $k=k(m,h,d)$. Let $H_1,\dots, H_k$ and $G_1,\dots, G_k$ and $G'_1,\dots,G'_k$ be given as above. In addition, by a slight abuse of notation, write $H_{k+1}=\overline{P_k}$ and $G_{k+1}=\partial P_k$. Then $H_{i+1}\subset \interior{H_i}$ for all $1\leq i\leq k$. Recall that $X\not\in\overline{P}$. On the one hand, by Lemma \ref{c-nested}, the trivial observation $\omega^{X_i}_{\Omega\setminus H_{i+1}}(G'_{i+1})\leq \omega^{X_i}_{\Omega\setminus H_{i+1}}(G_{i+1})$, and the fact that we are in Alternative 2, \begin{equation}\begin{split} \label{big-nest} \omega(P_k)&\leq \omega^{X}_{\Omega\setminus H_1}(G'_1)\prod_{i=1}^k\sup_{X_{i}\in G'_{i}\cap\Omega}\omega^{X_{i}}_{\Omega\setminus H_{i+1}}(G_{i+1})
\\ &\leq \omega^{X}_{\Omega\setminus H_1}(G'_1)\prod_{i=1}^k\left(1-\inf_{X_i\in G'_i\cap\Omega} \omega^{X_i}_{\Omega\setminus H_{i+1}}(A_i)\right) \leq \omega^{X}_{\Omega\setminus H_1}(G'_1)(1-\eta)^k.\end{split}\end{equation} To verify the final inequality, fix $Z\in G'_i$ and let $(Q,Q_*)$ be the admissible pair given by property (ii) in the definition of $G_i$. Then $\omega^{Z}_{\Omega\setminus H_{i+1}}(A_i)\geq \omega^{Z}_{\interior{Q}\setminus \Omega^c}(\Omega^c\cap\interior{Q}) \geq \eta$ by the maximum principle and \eqref{alt1}. On the other hand, \begin{equation}\label{first-land} \omega(P) \geq \omega_{\Omega\setminus H_1}(G'_1)\inf_{Z\in G'_1\cap\Omega} \omega_{\Omega}^{Z}(P)\geq \eta\, \omega^X_{\Omega\setminus H_1}(G'_1)\end{equation} by Lemma \ref{c-lower-bound}, the maximum principle, and \eqref{alt1}. Combining \eqref{big-nest} and \eqref{first-land}, we obtain \begin{equation}\omega(P_k)\leq \eta^{-1}(1-\eta)^k \omega(P).\end{equation} The consequence of this estimate is that if $k$ (hence $M$, hence $d$) is sufficiently large, then $\omega(P_k)$ is arbitrarily small relative to $\omega(P)$.

\begin{figure}\begin{center}\includegraphics[width=.6\textwidth]{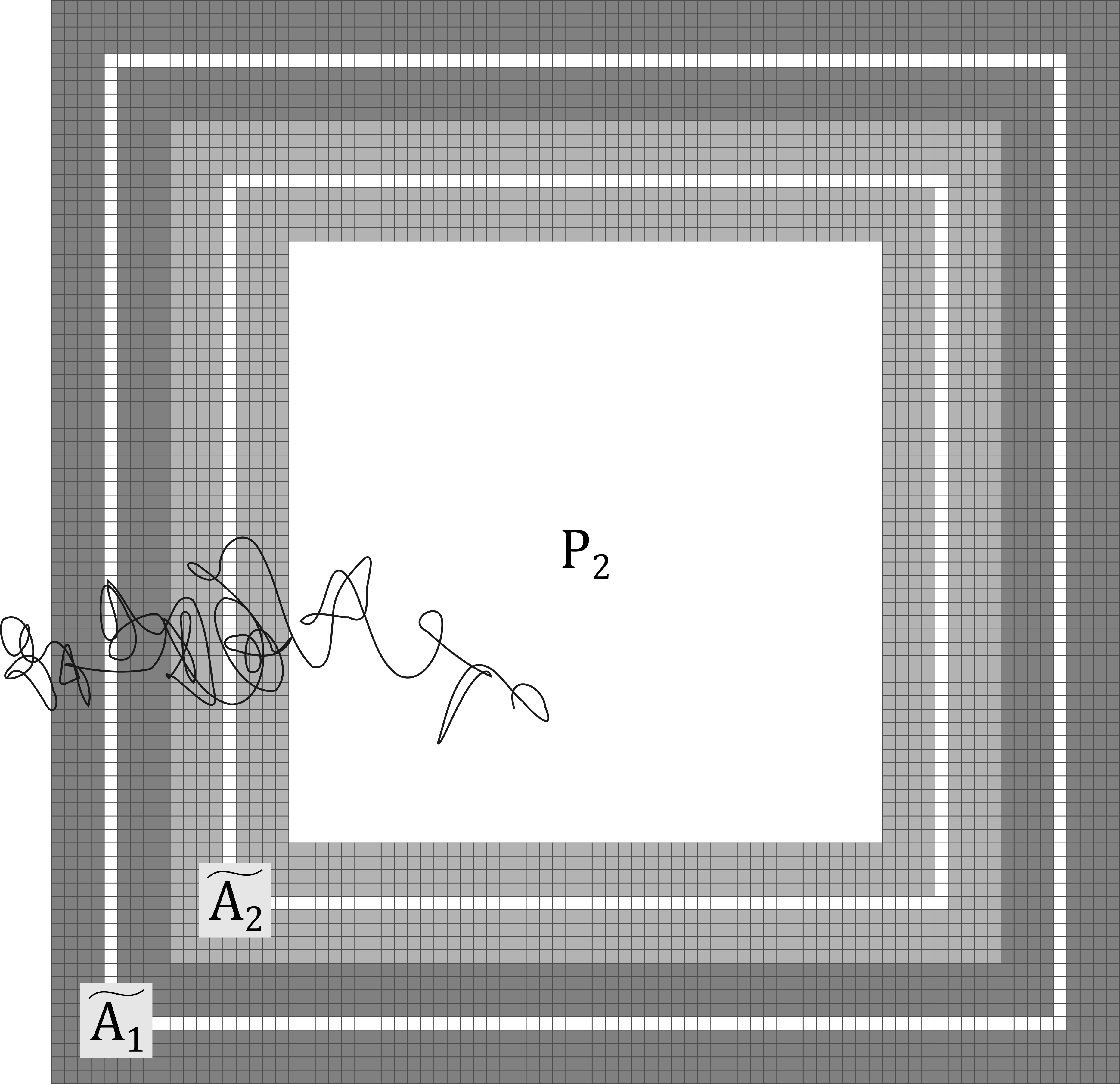}\end{center}\caption{A decomposition $P=\widetilde A_1\cup\dots \cup \widetilde A_h \cup P_{hm^{d-1}}$ when $n=2$, $m=9$, $h=2$, and $d=1$, where each little square represents a $d+1$ descendent of $P$. Brownian motion, started outside of $P$, cannot reach the inner region $P_{hm^{d-1}}$ without passing through surfaces $G_i$ (not displayed) drawn in the collars of white squares. Increasing $h$ raises the number of annuli $\widetilde A_j$ of children of $P$. Increasing $d$ yields a higher density $m^{d-1}$ of separating surfaces per annulus.}  \label{fig:paw}
\end{figure}

To proceed, note that each union $\widetilde A_j=A_{1+(j-1)m^{d-1}}\cup\dots\cup A_{m^{d-1}+(j-1)m^{d-1}}$ of $m^{d-1}$ consecutive rings $A_i$ of $d$-descendents is an annulus formed from children of $P$. That is, $\widetilde A_1$ is the outermost annulus of children, $\widetilde A_2$ is the second annulus of children, and so on. Assign $k=hm^{d-1}$, with $h$ given in the hypothesis of the theorem, and note that $k\leq M$, because $h<m/2$. Partition the set of all children of $P$ into two collections: $$\mathcal{A}=\{Q\in\Child(P):Q\subset A_1\cup\dots\cup A_k\}\quad\text{and}\quad\mathcal{B}=\{Q\in\Child(P):Q\subset P_k\}.$$ Writing $\vol(P_k)/\vol(P)=\delta$, we have  \begin{align*} \sum_{Q\in\Child(P)} \omega(Q)^{1/2}(\vol Q)^{1/2} &= \sum_{Q\in\mathcal{A}} \omega(Q)^{1/2}(\vol Q)^{1/2}
+\sum_{Q\in \mathcal{B}} \omega(Q)^{1/2}(\vol Q)^{1/2}\\
&\leq \omega(A_1\cup\dots\cup A_k)^{1/2}\left(\textstyle\sum_{Q\in \mathcal{A}}\vol Q\right)^{1/2}
+ \omega(P_k)^{1/2}(\vol P_k)^{1/2}\\
&\leq \left((1-\delta)^{1/2}+\delta^{1/2}\eta^{-1/2}(1-\eta)^{k/2}\right)\omega(P)^{1/2}(\vol P)^{1/2},\end{align*} where the first inequality holds by Cauchy-Schwarz. To find the value of $\delta$ in terms of the parameters $n$, $m$, and $h$, write $$\delta=\frac{\vol P_k}{\vol P} = \frac{(m^d-2k)^n}{m^{dn}}=\left(1-\frac{2k}{m^d}\right)^n
=\left(1-\frac{2h}{m}\right)^n.$$ We have shown that \begin{equation}\begin{split} \label{outcome2}
\sum_{Q\in\Child(P)} \omega(Q)^{1/2}(\vol Q)^{1/2} \leq m^{-\lambda}\omega(P)^{1/2}(\vol P)^{1/2},
\end{split}\end{equation} where $\lambda=-\log_m(\gamma)=-\log_m((1-(1-2h/m)^n)^{1/2} + (1-2h/m)^{n/2}\eta^{-1/2}(1-\eta)^{hm^{d-1}/2})>0$ by \eqref{gamma-def}.

\smallskip

\emph{Conclusion.} By \eqref{outcome0}, \eqref{outcome1}, and \eqref{outcome2}, the harmonic measure $\omega=\omega^X_{\Omega}$ of $\Omega$ satisfies the hypothesis of Lemma \ref{dim-lemma}. Therefore, $\overline{\dim}_H\,\omega \leq n-\lambda\rho/(\lambda+\rho)$. As we let $\Omega\subsetneq\RR^n$ be an arbitrary domain, this proves $b_n\geq \lambda\rho/(\lambda+\rho)$.
\end{proof}

\begin{remark}\label{V-Pi} Given $n$, $m$, $\eta$, $h$, and $d$, assign \begin{equation}V:=(1-(1-2h/m)^n)^{1/2}\quad\text{and}\quad \Pi:=\eta^{-1/2}(1-\eta)^{hm^{d-1}/2}.\end{equation} Then $\gamma<V+\Pi$, where $\gamma$ is defined in \eqref{gamma-def}. In particular, $V+\Pi< 1$ implies $\gamma<1$.\end{remark}

\section{Proof of Theorem \ref{main}} \label{sec:main}

Fix $n\geq 3$ (possibly, but not necessarily large). We first verify that $b_n>0$. Fix a large integer $m>\xi_m+2\sqrt{n}$ satisfying the stipulations below. By Corollary \ref{c-all-dim}, a Bourgain type estimate \eqref{alpha-bourgain} holds with $\alpha=C_n m^{n-2}$ and $\epsilon=1$. Set $\eta=1/(3\alpha)$, $h=1$, and $d=n$. Then certainly $\alpha\eta = 1/3 < (1-m^{-n})/(2-m^{-n})$. We claim the quantity $\gamma$ defined in \eqref{gamma-def} is less than $1$ if $m$ is sufficiently large. Indeed, on the one hand, $$V=(1-(1-2h/m)^n)^{1/2}=(1-(1-2/m)^n)^{1/2}<1/2$$ for large enough $m$. On the other hand,  $\Pi=\eta^{-1/2}(1-\eta)^{hm^{d-1}/2}=\eta^{-1/2}(1-\eta)^{m^{n-1}/2}$. Using the bound $\ln(1-x)\leq -x$ for all $0\leq x<1$, we see that $$\ln \Pi \leq \frac{1}{2}\ln(3C_n) + \frac{1}{2}(n-2)\ln(m) - \frac{m^{n-1}}{6C_n m^{n-2}}\rightarrow-\infty\quad\text{as $m\rightarrow\infty$}.$$ Hence $\Pi<1/2$ if $m$ is sufficiently large. Thus, $\gamma<V+\Pi<1$ if $m$ is sufficiently large. Therefore, $b_n>0$ by Theorem \ref{t:parameters}.

Now, suppose that $n\geq n_{1/3}$, where $n_{1/3}$ is given by Corollary \ref{c-large-dim}. Aiming for a quantitative lower bound on $b_n$ that is valid for all sufficiently large $n$, we may increase the value of $n$ as convenient. Set $m=n$. By Corollary \ref{c-large-dim}, a Bourgain type estimate \eqref{alpha-bourgain} holds with $\alpha=(1/3)(n\sqrt{2\pi e})^{n-2}$ and $\epsilon=1/2$. Set $\eta=(n\sqrt{2\pi e})^{-(n-2)}$, which ensures that $\alpha\eta=1/3<(1-m^{-n})/(2-m^{-n})$. Set $h=1$ and $d=2n-3$. On the one hand, $$V=(1-(1-2/n)^n)^{1/2}\sim(1-e^{-2})^{1/2}=0.9298....$$ On the other hand, $\Pi=\eta^{-1/2}(1-\eta)^{n^{2n-4}/2}$ satisfies $$\ln\Pi \leq \frac{1}{2}(n-2)\ln(n\sqrt{2\pi e})-\frac{1}{2}n^{2n-4}(n\sqrt{2\pi e})^{-(n-2)}=\frac{1}{2}(n-2)\ln(n\sqrt{2\pi e})-\frac{1}{2}\left(\frac{n}{\sqrt{2\pi e}}\right)^{n-2}.$$ As the latter expression tends to $-\infty$ as $n$ grows, we see that $\Pi<0.07$ for  large $n$. Thus, $\gamma<V+\Pi < 0.9998$ and $\lambda\geq -\ln(0.9998)/\ln(n)$ for all sufficiently large $n$. By Theorem \ref{t:parameters}, $b_n\geq \lambda\rho/(\lambda+\rho)$, where $\rho =0.914186(1-1/3)(1-n^{-n})n^{-2n(n-1)}/\ln(n)$. Since $\rho$ is substantially smaller than $\lambda$ for large $n$, it follows that $b_n\approx \rho$. In particular, since $0.914186(2/3)(1-n^{-n})>0.6$, we may conclude that $b_n\geq 0.6 n^{-2n(n-1)}/\ln(n)$ for all sufficiently large $n$. Therefore, since $b_n>0$ for all $n\geq 3$, the theorem holds: there exists $c>0$ such that $b_n\geq c\, n^{-2n(n-1)}/\ln(n)$ for all $n\geq 3$.

\begin{remark}For the large $n$ case, one could also choose $d=n-1+\theta(n-2)$ for any $\theta>0$ by making $n$ large enough depending on $\theta$. For simplicity, we chose $\theta=1$.\end{remark}

\section{Proof of Theorem \ref{lowdim}} \label{sec:low}

When $n=3$ and $m=9$, a Bourgain type estimate \eqref{alpha-bourgain} holds with $\alpha=60.8979$ and $\epsilon=0.000001$ by Corollary \ref{c-3}. Assign $\eta=0.0046$, $h=3$, and $d=4$. With the aid of a calculator, one can see that $\gamma=(1-(1-2h/m)^n)^{1/2} + (1-2h/m)^{n/2}\eta^{-1/2}(1-\eta)^{hm^{d-1}/2}<0.9996$. Thus, by Theorem \ref{t:parameters}, $b_3 \geq \lambda\rho/(\lambda+\rho)\geq 1.452...\times 10^{-15}$. See the appendix for details.

When $n=4$ and $m=11$, a Bourgain type estimate \eqref{alpha-bourgain} holds with $\alpha=1660.53$ by Corollary \ref{c-4}. When $\eta=0.00026$, $h=4$, and $d=5$, one may check that $\gamma<0.9995$. Thus, by Theorem \ref{t:parameters}, $b_4\geq \lambda\rho/(\lambda+\rho)\geq 2.199...\times 10^{-26}$. Once again, see the appendix.

\begin{remark} The authors do not claim that these bounds are sharp, but do believe that they are likely close to what the method can prove without further improvements to \eqref{better-bourgain} or a more complicated case analysis in the proof of Theorem \ref{t:parameters}. Another small optimization available is to use H\"older's inequality with conjugate exponents depending on $m,\eta,h,d$ instead of the Cauchy-Schwarz inequality and $p=q=1/2$ in the statement and proof of Lemma \ref{dim-lemma} and in the definition of $\gamma$ and proof of Theorem \ref{t:parameters}.\end{remark}

\section{Coda}\label{sec:coda}

The story is far from over. Now that explicit lower bounds on $b_3$ and $b_4$ and asymptotic lower bounds on $b_n$ are known, one can test new methods and estimates against Bourgain's method. The authors invite further activity to improve their estimates on (or compute!) the dimension of harmonic measure in $\RR^n$, $n\geq 3$.

\appendix

\section{Wolfram Language code for estimating \texorpdfstring{$b_3$}{b\textunderscore 3}}

\captionsetup{justification=centering}

\begin{table}\caption{Bounding Bourgain's constant for harmonic measure: $b_n\geq \lambda\rho/(\lambda +\rho)\approx\rho$ when $\rho\ll\lambda$.} \label{table-of-computations}
\begin{tabular}{ccccccccc}
  \toprule
  $n$ & $m$ & $\eta$ & $h$ & $d$ & $\alpha$ & $\gamma$ & $\lambda$ & $\rho$ \\ \toprule
  3 & 5 & 0.0005 & 2 & 7 & 303.102 & 0.9976... & $1.488...\times 10^{-3}$ & $8.020...\times 10^{-18}$ \\
  3 & 6 & 0.0008 & 2 & 6 & 277.560 & 0.9947... & $2.911...\times 10^{-3}$ & $1.801...\times 10^{-17}$ \\
  3 & 7 & 0.0019 & 3 & 5 & 83.8178 & 0.9998... & $7.481...\times 10^{-5}$ & $2.418...\times 10^{-16}$ \\
  3 & 8 & 0.0011 & 3 & 5 & 81.9976 & 0.9965... & $1.678...\times 10^{-3}$ & $2.215...\times 10^{-17}$ \\
  \color{blue} 3 & \color{blue} 9 & \color{blue} 0.0046 & \color{blue} 3 & \color{blue} 4 & \color{blue} 60.8979 & \color{blue} 0.9996... & \color{blue} $1.616...\times 10^{-4}$ & \color{blue} $1.452...\times 10^{-15}$ \\
  3 & 10 & 0.0031 & 4 & 4 & 61.4480 & 0.9992... & $3.385...\times 10^{-4}$ & $3.210...\times 10^{-16}$ \\
  3 & 11 & 0.0022 & 4 & 4 & 54.2657 & 0.9984... & $6.516...\times 10^{-4}$ & $8.031...\times 10^{-17}$ \\
  3 & 12 & 0.0016 & 5 & 4 & 55.5835 & 0.9993... & $2.254...\times 10^{-4}$ & $2.174...\times 10^{-17}$ \\
  3 & 13 & 0.0012 & 5 & 4 & 52.5339 & 0.9982... & $6.978...\times 10^{-4}$ & $6.521...\times 10^{-18}$ \\
  3 & 14 & 0.0009 & 5 & 4 & 54.1918 & 0.9988... & $4.385...\times 10^{-4}$ & $2.117...\times 10^{-18}$ \\
  \midrule
  4 & 7 & 0.00006 & 3 & 7 & 2409.54 & 0.9998... & $7.291...\times 10^{-5}$ & $3.637...\times 10^{-28}$ \\
  4 & 8 & 0.00016 & 3 & 6 & 2425.26 & 0.9999... & $2.780... \times 10^{-5}$ & $1.390...\times 10^{-26}$ \\
  4 & 9 & 0.00009 & 3 & 6 & 1813.48 & 0.9978... & $9.801...\times 10^{-4}$ & $6.651...\times 10^{-28}$ \\
  4 & 10 & 0.00005 & 4 & 6 & 1834.77 & 0.9994... & $2.361...\times 10^{-4}$ & $3.605...\times 10^{-29}$ \\
  \color{blue} 4 & \color{blue} 11 & \color{blue} 0.00026 & \color{blue} 4 & \color{blue} 5 & \color{blue} 1660.53 & \color{blue} 0.9995... & \color{blue} $2.062...\times 10^{-4}$ & \color{blue} $2.199...\times 10^{-26}$ \\
  4 & 12 & 0.00017 & 5 & 5 & 1685.89 & 0.9999... & $2.779...\times 10^{-5}$ & $3.301...\times 10^{-27}$ \\
  4 & 13 & 0.00012 & 5 & 5 & 1619.82 & 0.9995... & $1.932...\times 10^{-4}$ & $5.289...\times 10^{-28}$ \\
  4 & 14 & 0.00009 & 5 & 5 & 1649.02 & 0.9981... & $6.908...\times 10^{-4}$ & $9.177...\times 10^{-29}$ \\
  4 & 15 & 0.00006 & 6 & 5 & 1626.75 & 0.9997... & $8.531...\times 10^{-5}$ & $1.809...\times 10^{-29}$ \\
  4 & 16 & 0.00005 & 6 & 5 & 1659.76 & 0.9985... & $5.340...\times 10^{-4}$ & $3.816...\times 10^{-30}$ \\
   \bottomrule
\end{tabular}\end{table}

We wrote the following code in Mathematica 13 to estimate $b_3$ using Theorem \ref{t:parameters} and Corollary \ref{c-3}. See Table \ref{table-of-computations} for a record of outputs. For each fixed $m\geq 5$, the parameters $\eta$, $h$, and $d$ were optimized by hand. To maximize $\rho$, the first priority is to minimize $d$. To rule out small values of $d$, take $\eta\approx \alpha^{-1}(1-m^{-3})/(2-m^{-3})$ and check that $\gamma>1$ for each integer $1\leq h<m/2$. Once the optimal value of the integer $d$ is identified, the second priority is to minimize the real-valued parameter $\eta$. Using the current best guess for $\eta$ (keeping $\gamma<1$), adjust $h$ to minimize $\gamma$. One can then test the value of $\gamma$ against smaller values of $\eta$. If $\gamma<1$ for some smaller value of $\eta$, update the best guess for $\eta$ and repeat (adjust $h$, test smaller values of $\eta$). Halt the search for $\eta$ once all smaller values of $\eta$ (up to some predetermined number of decimals) yield $\gamma>1$. Use the values of $\lambda$ and $\rho$ associated to $m$, $\eta$, $h$, and $d$ to bound $b_3$ from below by $\lambda\rho/(\lambda+\rho)$.

\small

\begin{verbatim}
(* All formulas use n=3, epsilon=0.000001 *)
bgAlpha[m_] := (LHS = 1/Sqrt[3] - 2/(m - 2 + Mod[m,2]);
   RHS = m*2/Sqrt[3] + 27*Pi*m^(-0.999999)/(1 - m^(-1.999999));
   Ceiling[10000*RHS/LHS]/10000); (* round up fourth decimal *)

bgMaxEta[m_] := ((1 - m^(-3))/(2 - m^(-3)))/bgAlpha[m];

bgV[m_,h_] := If[h<m/2, (1 - (1 - 2*h/m)^3)^0.5, 1];

bgEtaProd[m_,eta_,h_,d_] := If[h<m/2,
   (1 - 2*h/m)^1.5 * eta^(-0.5) * ((1-eta)^(0.5*h*m^(d-1))), 1];
(* If h >= m/2, then bgV[m,h]=1 and bgEtaProd[m,eta,h,d]=1 *)

bgGamma[m_,eta_,h_,d_] := bgV[m,h]+ bgEtaProd[m,eta,h,d];

bgLambda[m_,eta_,h_,d_] := Max[0,-Log[m,bgGamma[m,eta,h,d]]];
(* If bgGamma[m,eta,h,d]>=1, then bgLambda[m,eta,h,d]=0 *)

bgRho[m_,eta_,d_]:= If[eta <= bgMaxEta[m],
   0.914186*(1 - m^(-3))*(1 - bgAlpha[m]*eta)*m^(-3*(d + 1))/Log[m], 0];
(* If eta > bgMaxEta[m], then bgRho[m,eta,d]=0 *)

bgLowerBound[m_,eta_,h_,d_] := (lambda = bgLambda[m,eta,h,d];
   rho = bgRho[m, eta, d];
   lambda*rho/(lambda + rho));
(* Returns lower bound on b_3 for any admissible (m,eta,h,d) *)
(* For example, bgLowerBound[9,0.0046,3,4] returns 1.45271 * 10^(-15) *)
\end{verbatim}

The code used to estimate $b_4$ (omitted) is similar. It can be reproduced by modifying the definition of \texttt{bgAlpha} using Corollary \ref{c-4} instead of Corollary \ref{c-3} and changing $n=3$ to $n=4$ in the definitions of \texttt{bgMaxEta}, \texttt{bgV}, \texttt{bgEtaProd}, and \texttt{bgRho}.

\normalsize

\bibliography{cdim}
\bibliographystyle{amsbeta}

\end{document}